\documentclass[12pt]{amsart}
  \usepackage{latexsym} 
  \usepackage[all]{xy}
  \usepackage{amsfonts} 
  \usepackage{amsthm} 
  \usepackage{amsmath} 
  \usepackage{amssymb}
  \usepackage{pifont}  
  \usepackage{enumerate}
  \xyoption{2cell}
\usepackage{color}

\usepackage{overpic}

\usepackage{tikz}

\usetikzlibrary{arrows,backgrounds,decorations}
\usetikzlibrary{positioning,automata}

  \def\<{{\langle}} 
  \def\>{{\rangle}}

  \def\note#1{{}}

  \def\note#1{} 

   \def\cK{{\mathcal K}}

  \def\beq{\begin{equation}} 
  \def\eeq{\end{equation}}

 \def\coker{\mathrm{coker}}

  \newcounter{zlist}

  \newcounter{blist}

  \newcounter{rlist}

\def\stac#1{\raise-.2cm\hbox{$\stackrel{\displaystyle\otimes}{\scriptscriptstyle{#1}}$}}

\def\cten#1{\raise-.2cm\hbox{$\stackrel{\displaystyle\widehat{\otimes}}
{\scriptscriptstyle{#1}}$}}

  \def\Label#1{\label{#1}\ifmmode\llap{[#1] }\else 
  \marginpar{\smash{\hbox{\tiny [#1]}}}\fi} 
  \def\Label{\label}

  \newtheorem{proposition}{Proposition}[section]
   
  \newtheorem{corollary}[proposition]{Corollary} 
  \newtheorem{theorem}[proposition]{Theorem} 

  \theoremstyle{definition}

  \newtheorem{example}[proposition]{Example}

  \theoremstyle{remark} 
  \newtheorem{remark}[proposition]{Remark}

  \newcounter{c} 
   
  \newcommand{\etyk}[1]{\vspace{-7.4mm}$$\begin{equation}\Label{#1} 
  \addtocounter{c}{1}} 
  \renewcommand{\]}{\ifnum \value{c}=1 $$\else \end{equation}\fi} 
\def\CC{{\mathbb C}}

\def\NN{{\mathbb N}}

\def\PP{{\mathbb P}}

\def\TT{{\mathbb T}}

\def\WW{{\mathbb W}}

\def\ZZ{{\mathbb Z}}

\newcommand{\Cc}{\mathcal{C}}

\def\m{{\sf m}}

\def\*C{{}^*\hspace*{-1pt}{\Cc}}

\def\text#1{{\rm {\rm #1}}}

 \def\y{\mathbf{y}}

 \def\1{\mathbf{1}}

\begin{document}

\title[Quantum lens  and weighted projective spaces]{The $C^*$-Algebras of Quantum Lens \\ and Weighted Projective Spaces}  
 \author{Tomasz Brzezi\'nski}
 \address{ Department of Mathematics, Swansea University, Singleton Park, Swansea SA2 8PP, U.K.  
\& Department of Mathematics, University of Bia\l ystok, K.\ Cio\l kowskiego 1M, 15--245 Bia\l ystok, Poland} 
  \email{T.Brzezinski@swansea.ac.uk}   
\author{Wojciech Szyma\'nski}
\address{Department of Mathematics and Computer Science, University of Southern Denmark, 
Campusvej 55, 5230 Odense M, Denmark} 
\email{szymanski@imada.sdu.dk} 

\date{\today}
 

\begin{abstract} 
It is shown that the algebra of continuous functions on the  quantum $2n+1$-dimensional lens space $C(L^{2n+1}_q(N; m_0,\ldots, m_n))$ is a graph $C^*$-algebra, for arbitrary positive weights 
$ m_0,\ldots, m_n$. The form of the corresponding graph is determined from the skew product  of the graph which defines the algebra of continuous functions on the quantum sphere $S_q^{2n+1}$ and the cyclic group $\ZZ_N$, with the labelling induced by the weights. Based on this description, the $K$-groups of specific examples are computed. Furthermore, the $K$-groups of the algebras of continuous functions on quantum weighted projective spaces $C(\WW\PP_q^n(m_0,\ldots, m_n))$, interpreted as fixed points under the circle action on $C(S_q^{2n+1})$, are computed under a mild assumption on the weights.
\end{abstract} 

\maketitle 
  
\section{Introduction}

The aim of the present paper is investigation of noncommutative $C^*$-algebras of continuous functions on quantum deformations of two classes of (possibly singular) spaces, namely the lens and weighted projective spaces. Central to these studies is identification of the algebras in question as graph $C^*$-algebras or subalgebras of graph $C^*$-algebras corresponding to certain actions of the circle group. 

In classical geometry both (weighted) lens and weighted projective spaces are obtained as quotients of groups acting (with possible different positive integer weights) on the odd dimensional spheres. In the former case the  acting group is a finite cyclic group, in the latter it is the circle group. Depending on the choice of weights (relatively to the order of the acting group) the action can be free or almost free, thus leading to most easily accessible examples of {\rm orbifolds} in the latter case, \cite{Thu:geo} or \cite{AdeLei:orb}. The fact that these orbifolds are defined by explicit group actions on the spheres suggests an accessible way of defining and studying their noncommutative counterparts through the exploration of analogous actions on quantum odd dimensional spheres,  \cite{VakSob:alg}. On this premise, quantum lens spaces were defined in \cite{HonSzy:len}, and more recently weighted projective spaces were introduced in \cite{BrzFai:tea} with a specific aim of studying quantum or noncommutative orbifolds,  a task subsequently undertaken for example in \cite{Har:orb}, \cite{DAnLan:wei},  \cite{SitVen:geo}, \cite{Har:Dir} and a series of papers by the first author of the present paper. Surprising results of these initial studies include observations that, upon deformation, classically non-free actions become free (see e.g.\ \cite{BrzFai:not}, \cite{AriDAn:Pim} or \cite{DAnLan:wei}) and deformations of non-smooth objects behave as if they were deformations of smooth objects (see e.g.\ \cite{Brz:smo}, \cite{BrzSit:smo}). Despite a significant progress in understanding the structure of weighted projective spaces in special cases (see e.g.\ \cite{DAnLan:wei}, \cite{BrzFai:not}) and deformations of classically singular lens spaces, the full picture is still not complete and does not include some important classes of the objects in question. Guided by the experience of working with quantum lens spaces that correspond to classically non-singular case, \cite{HonSzy:len}, graph $C^*$-algebras appear to offer an effective tool to fill in this gap in our understanding of the structure of the relevant $C^*$-algebras and their 
$K$-theoretic invariants. We are exploiting this opportunity here.
 
A directed graph $G = (G^0,G^1, \varrho,\sigma)$ consists of two sets $G^0$ and $G^1$ (the former the set of vertices and  the latter the set of edges) and two mappings $\varrho, \sigma: G^1\to  G^0$, called the range and source,  respectively. Given a graph $G$ with countably many vertices and edges, $C^*(G)$ denotes the  $C^*$-algebra defined as follows, \cite{FowLac:alg}. $C^*(G)$ is the universal $C^*$-algebra generated by a set  $\{P_v \; |\; v\in G^0\}$ of mutually orthogonal projections and a set $\{S_e \; |\; e\in G^1\}$
of partial isometries   which satisfy the following relations, for all edges $e\neq f \in G^1$ and all vertices $v\in G^0$ emitting a finite number of edges, 
\begin{subequations}\label{graph}
\begin{gather}
S^*_eS_f = 0, \qquad S^*_e S_e = P_{\varrho(e)}, \qquad S_e S^*_e \leq P_{\sigma(e)}, \label{ss}\\
P_v = \sum_{e\in G^1\; :\;  \sigma(e) =v} S_e S^*_e. \label{p}
\end{gather}
\end{subequations}

Graph $C^*$-algebras include important classes of operator algebras, such as the Cuntz-Krieger  algebras \cite{CunKri:cla} or AF algebras. Significant advantage of working with graph $C^*$-algebras stems from the ease 
with which one can calculate their $K$-theory and  primitive ideal spectrum. This feature 
of graph $C^*$-algebras has been widely exploited in their applications  to the classification programme of general $C^*$-algebras (for example, see \cite{Szy:k}, \cite{Sp:s}, \cite{ET:c} or \cite{AE:d}). 
In addition, they have influenced recent developments in purely algebraic ring theory, leading to the introduction of Leavitt path algebras \cite{AbrAra:Lea} in an attempt to explore their classification power beyond operator algebra theory; see \cite{Abr:Lea} for an illuminating  review.   More importantly from the point of view of the subject matter of this text, algebras of continuous functions on quantum spheres and on deformations of non-singular lens spaces can also be interpreted as graph $C^*$-algebras  \cite{HonSzy:sph}, \cite{HonSzy:len}. We extend this interpretation to quantum deformations of all (weighted) lens spaces, including those that are classically singular, and employ it to compute the K-theory of a fairly general class of quantum weighted projective spaces and of particular examples of (weighted, singular) quantum lens spaces.

The paper is organised as follows. In Section~\ref{sec.lens}, we first recall the algebraic definition of quantum lens spaces. The coordinate algebras of quantum lens  spaces are defined as fixed points of  the weighted action of the cyclic group $\ZZ_N$ on (the generators of) the coordinate algebra of the quantum odd-dimensional sphere. We make no assumption on the existence of common factors of weights and the order of the group. Next, using the identification of the algebra of continuous functions on the quantum odd dimensional sphere with the graph $C^*$-algebra associated to a graph $L_{2n+1}$ \cite{HonSzy:sph}, we extend the cyclic group action to the action on this algebra. The resulting fixed point algebra has been shown in \cite{HonSzy:len} to be a graph $C^*$-algebra, provided all weights are coprime with the order of the acting group. We extend this identification to all weights, thus relaxing the coprimeness assumption. Similarly to \cite{HonSzy:len} we construct a suitable graph $L_{2n+1}^{N;\m}$ by relating it to the skew product graph $L_{2n+1}\times_c \ZZ_N$ (where the labelling $c$ is determined by the weights), and using the result of Crisp \cite{Cri:cor} that fixed points of a finite group action on a graph $C^*$-algebra can be identified with specific corner of the algebra associated to the skew product graph. The construction of $L_{2n+1}^{N;\m}$ explores the values of weights modulo the order of the cyclic group and thus heavily depends on them; we illustrate this by a series of examples. The identification of  the algebras of continuous functions on  quantum lens spaces as algebras associated to explicitly described graphs allows one for more effective calculation of their $K$-groups, in particular the $K_0$-groups; we illustrate it be a series of examples too.

In  Section~\ref{sec.proj}, we study algebras of continuous functions on quantum weighted projective spaces $\WW\PP_q^n(m_0,\ldots, m_n)$. On the algebraic level these are defined as fixed points of weighted circle group actions on the quantum odd dimensional sphere, \cite{BrzFai:tea}. On the other hand they can also be identified with a free or principal \cite{BrzHaj:Che} action of the circle group on quantum lens spaces such that all weights divide the order of the cyclic group, \cite{BrzFai:not}. Both actions can be lifted to actions on continuous functions on quantum spheres and lens spaces. The analysis of algebraic structure of quantum weighted projective spaces, in particular of deriving generators and relations, is notoriously difficult, as one has to deal not only with an increasing number of generators but also with relative divisibility properties of the weights.  Until now, even in the lowest dimensional case, $n=1$, the algebraic and operator algebraic structure of $\WW\PP_q^1(m_0, m_1)$ has been understood completely only in the case of coprime weights \cite{BrzFai:tea}. In higher dimensions, the full list of  (algebraic) generators of the coordinate algebra $C(\WW\PP_q^n(m_0,\ldots, m_n))$  is given in \cite{DAnLan:wei}, in the case the weights of the form $m_i = \prod_{j\neq i} l_j$, where $l_0, \ldots , l_n$ are pairwise coprime integers. Furthermore, such a list of generators is given in \cite{BrzFai:not} provided all the weights but the last one are equal to 1. We  prove that the algebra of continuous functions on $\WW\PP_q^1(m_0, m_1)$ is an AF graph $C^*$-algebra, and compute its $K$-theory with no restrictions on the weights (Proposition~\ref{3dimweighted}). It turns out that $C(\WW\PP_q^1(m_0, m_1))$ does not depend on the actual values of the weights, but only on $m_1$ divided by its greatest common divisor with $m_0$, and hence, as a topological noncommutative space any quantum projective line $\WW\PP_q^1(m_0, m_1)$  is isomorphic to the quantum teardrop $\WW\PP_q^1(1, m)$ (with $m=m_1/\gcd(m_0,m_1)$).  Finally we derive a short exact sequence which characterises quantum weighted projective spaces with weights $m_0,\ldots ,m_{n-1}$ coprime with $m_n$, and use it to compute their $K$-theory in the case the weights have the property that for each $j\geq 1$ there is an $i<j$ so that $m_i$ and $m_j$ are relatively prime.

\section{Quantum weighted lens spaces as graph $C^*$-algebras}\label{sec.lens}\setcounter{equation}{0}

In this section we prove that noncommutative algebras of continuous functions on all quantum (weighted) lens spaces are graph $C^*$-algebras.

The algebra of continuous functions on the quantum odd-dimensional sphere $C(S^{2n+1}_q)$ is defined as the universal $C^*$-algebra with generators $z_0,z_1,\ldots ,z_n$, subject to the following relations:
\begin{subequations}\label{sph}
\begin{gather}
z_iz_j = qz_jz_i \quad \mbox{for $i<j$}, \qquad z_iz^*_j = qz_j^*z_i \quad \mbox{for $i\neq j$}, \label{sph1}\\
z_iz_i^* = z_i^*z_i + (q^{-2}-1)\sum_{j=i+1}^n z_jz_j^*, \qquad \sum_{j=0}^n z_jz_j^*=1, \label{sph2}
\end{gather}
\end{subequations}
where $q$ is a real number, $q\in (0,1)$; see \cite{VakSob:alg}. As explained in \cite{HonSzy:sph}, the algebra $C(S^{2n+1}_q)$ can be interpreted as a $C^*$-algebra associated to a graph $L_{2n+1}$ defined as follows. $L_{2n+1}$ has $n+1$ vertices $v_0, v_1, \ldots, v_n$, and $(n+1)(n+2)/2$ edges $e_{ij}$, $i=0,\ldots ,n$, $j = i, \ldots ,n$, with $v_i$ the source and $v_j$ the range of $e_{ij}$. 

Let us fix a sequence of  positive integers $\m := m_0,\ldots, m_n$. For any natural number $N$, $C(S^{2n+1}_q)$ admits the action of the cyclic group $\ZZ_N$ defined by 
\begin{equation} \label{action}
\varrho_\m^N: z_i\mapsto \zeta^{m_i} z_i, 
\end{equation}
where $\zeta$ is a generator of $\ZZ_N$. Under the isomorphism $C(S^{2n+1}_q)\cong 
C^*(L_{2n+1})$, this action takes the form 
\begin{equation} \label{actionongraph}
\varrho_\m^N: S_{e_{ij}}\mapsto \zeta^{m_i} S_{e_{ij}}, \;\;\; 
\varrho_\m^N: P_{v_i} \mapsto P_{v_i}. 
\end{equation}
The fixed points of this action form the algebra of continuous functions on the quantum lens space $C(L^{2n+1}_q(N; \m))$. 
It is shown in \cite{HonSzy:len} that $C(L^{2n+1}_q(N; \m))$ is a graph $C^*$-algebra provided all the $m_i$ are coprime with $N$. We extend this result to the general case with arbitrary weight vector $\m$, below. 

As a matter of fact, $C(L_q^{2n+1}(N; \m))$ is isomorphic to the full corner of the graph $C^*$-algebra associated to the skew product graph $L_{2n+1}\times_c {\ZZ_N}$, where the labelling $c$ is induced from the $\ZZ_N$-action $\varrho^N_\m$, namely $c: e_{ij}\mapsto m_i \! \mod \! N$. More explicitly, the graph $L_{2n+1}\times_c {\ZZ_N}$ has vertices $(v_i, r)$, $i=0,\ldots, n$, $r=0,\ldots , N-1$ and edges $(e_{ij},r)$, $i,j = 0,\ldots n$, $i\leq j$, $r = 0,\ldots N-1$, with $(v_i, r-m_i \!\mod\! N)$ being the source and $(v_j,r)$ being the range of  $(e_{ij},r)$. 
For example, the skew product graph corresponding to $n=1$, $N=6$, $m_0=1$, $m_1=3$, comes out as
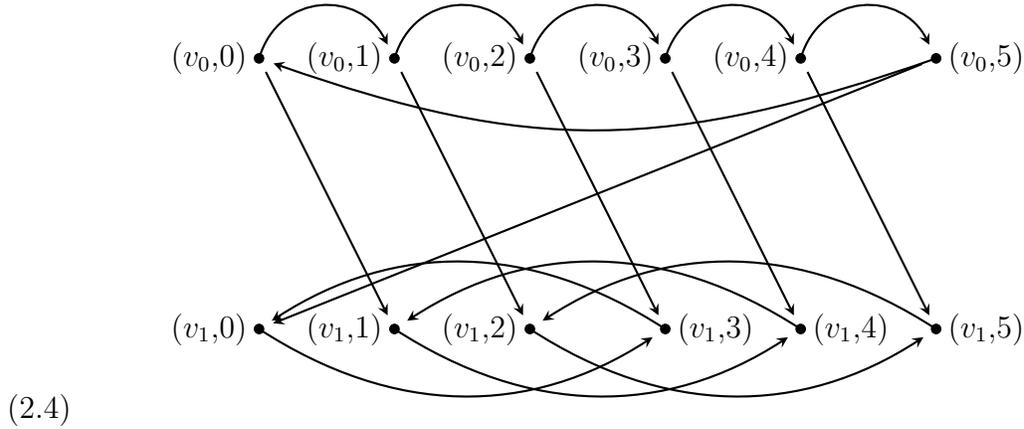
\begin{equation}
\begin{tikzpicture}[scale=1.8]

\node (0_0) at (0,0) [circle] {};
\node (1_0) at (1,0) [circle] {};
\node (2_0) at (2,0) [circle] {};
\node (3_0) at (3,0) [circle] {};
\node (4_0) at (4,0) [circle] {};
\node (5_0) at (5,0) [circle] {};
\node (0_1) at (0,2) [circle] {};
\node (1_1) at (1,2) [circle] {};
\node (3_1) at (3,2) [circle] {};
\node (2_1) at (2,2) [circle] {};
\node (4_1) at (4,2) [circle] {};
\node (5_1) at (5,2) [circle] {};

\foreach \x in {0,1,2,3,4,5} \filldraw  (\x_0) circle (1pt);
\foreach \x in {0,1,2,3,4,5} \filldraw  (\x_1) circle (1pt);

\draw (-0.37,0) node {$(v_1,\!0)$};
\draw (0.63,0) node {$(v_1,\!1)$};
\draw (1.63,0) node {$(v_1,\!2)$};
\draw (3.37,0) node {$(v_1,\!3)$};
\draw (4.37,0) node {$(v_1,\!4)$};
\draw (5.37,0) node {$(v_1,\!5)$};
\draw (-0.37,2) node {$(v_0,\!0)$};
\draw (0.63,2) node {$(v_0,\!1)$};
\draw (1.63,2) node {$(v_0,\!2)$};
\draw (2.63,2) node {$(v_0,\!3)$};
\draw (3.63,2) node {$(v_0,\!4)$};
\draw (5.37,2) node {$(v_0,\!5)$};

\draw[-stealth,thick] (0_1)-- (1_0) node[pos=0.5, inner sep=0.5pt, anchor=west] {};
\draw[-stealth,thick] (1_1)-- (2_0) node[pos=0.5, inner sep=0.5pt, anchor=west] {};
\draw[-stealth,thick] (2_1)-- (3_0) node[pos=0.6, inner sep=0.5pt, anchor=south] {};
\draw[-stealth,thick] (3_1)-- (4_0) node[pos=0.5, inner sep=0.5pt, anchor=west] {};
\draw[-stealth,thick] (4_1)-- (5_0) node[pos=0.5, inner sep=0.5pt, anchor=west] {};
\draw[-stealth,thick] (5_1)-- (0_0) node[pos=0.6, inner sep=0.5pt, anchor=south] {};

\draw[-stealth,thick] (0,0) .. controls (1,-.65) and (2,-.65) ..  (3_0);

\draw[-stealth,thick] (1,0) .. controls (2,-.65) and (3,-.65) .. (4_0);

\draw[-stealth,thick] (2,0) .. controls (3,-.65) and (4,-.65) .. (5_0);

\draw[-stealth,thick] (3,0) .. controls (2,.65) and (1,.65) .. (0_0);

\draw[-stealth,thick] (4,0) .. controls (3,.65) and (2,.65) .. (1_0);

\draw[-stealth,thick] (5,0) .. controls (4,.65) and (3,.65) .. (2_0);

\draw[-stealth,thick] (0,2) .. controls (0.15,2.5) and (.75,2.5) .. (1_1);

\draw[-stealth,thick] (1,2) .. controls (1.15,2.5) and (1.75,2.5) .. (2_1);

\draw[-stealth,thick] (2,2) .. controls (2.15,2.5) and (2.75,2.5) .. (3_1);

\draw[-stealth,thick] (3,2) .. controls (3.15,2.5) and (3.75,2.5) .. (4_1);

\draw[-stealth,thick] (4,2) .. controls (4.15,2.5) and (4.75,2.5) .. (5_1);

\draw[-stealth,thick] (5,2) .. controls (3,1.3) and (2,1.3) .. (0_1);
\end{tikzpicture}
\end{equation}

Based on the skew product graph $L_{2n+1}\times_c {\ZZ_N}$ we construct a graph $L_{2n+1}^{N;\m}$ in the following way. $L_{2n+1}^{N;\m}$ has vertices $v_i^r$, $i=0,\ldots ,n$, $r=0,\ldots ,\gcd(N,m_i) -1$, and edges $e^{rs}_{ij;a}$ with the source $v_i^r$ and the range $v_j^s$, and labelled additionally by $a = 1,\ldots , n_{ij}^{rs}$, where $n_{ij}^{rs}$ is a number of paths in $L_{2n+1}\times_c {\ZZ_N}$  from $(v_i,r)$ to $(v_j,s)$ that do not pass through $(v_k,t)$, with $k=i+1, \ldots , j-1$ and $t= 0, \ldots , \gcd(m_k, N) -1$; such paths are termed {\em admissible}. Since there are no edges $(e_{ij}, r)$ in $L_{2n+1}\times_c {\ZZ_N}$ with $i>j$, one may assume that $i\leq j$ in $e^{rs}_{ij;a}$. We refer to the index $i$ as labelling the {\em levels}, and to index $r$ as labelling the {\em loops} in $L_{2n+1}^{N;\m}$. 

It is helpful to analyse the graphs $L_{2n+1}\times_c {\ZZ_N}$ and $L_{2n+1}^{N;\m}$ more closely. Define
\begin{equation}\label{cd}
c_i = \gcd(N,m_i), \qquad d_i = \frac{N}{c_i},
\end{equation}
and observe that $d_i$ is coprime with $m_i/c_i$. The admissible paths from $(v_i, r)$ to $(v_j,s)$ with $r\in \{0, 1,\ldots ,c_i -1\}$, $s\in\{0, 1,\ldots ,c_j -1\}$ have the form
\begin{equation}\label{admissible}
(e_{ii_1}, r+m_i)(e_{i_1i_2}, r_1)\ldots (e_{{i_k}i_{k+1}}, r_k)(e_{i_{k+1}j},s),
\end{equation}
where 
\begin{equation}\label{cond}
r_t = r+m_i +\sum_{a=1}^tm_{i_a}\geq c_{i_t}, \quad t = 1,\ldots , k, \qquad r_k +m_{i_{k+1}}=s 
\end{equation}
(all sums are computed modulo $N$), and no vertex appears twice as the range (or, equivalently, source) of any of the edges that compose into the path \eqref{admissible}. There are no admissible paths between $(v_i,r)$ and $(v_i,s)$ if $r \neq s$. Indeed if there were such a path, then there would exist integers $a,b$ such that 
\begin{equation}\label{coprime}
r - s = am_i -bN.
\end{equation}
The right hand side is divisible by $c_i$, while the left hand side is not, since $|r-s| <c_i$, which gives the desired contradiction. On the other hand, there is exactly one path connecting $(v_i,r)$ with itself:
\begin{equation}\label{loop}
(e_{ii}, r+m_i)(e_{ii},r+2m_i) \ldots (e_{ii}, r+(d_i-1)m_i)(e_{ii}, r).
\end{equation}
To see that \eqref{loop} is admissible, first note that
$$
r + d_i m_i = r + \frac{N}{c_i} m_i = r+  \frac{m_i}{c_i} N \equiv r \mod N,
$$
so that the condition in \eqref{cond} is satisfied.  Furthermore there are no  edges in  \eqref{loop} with the same source. Otherwise, there would need to exist integers $a,b \in {0, \ldots , d_i -1}$ such that $r + am_i = r + bm_i$ modulo $N$. This would imply the existence of $c\in \ZZ$ such that $(a-b)m_i = cN$ or, when both sides are divided by $c_i$,
\begin{equation}\label{contra}
(a-b) \frac{m_i}{c_i} = cd_i.
\end{equation}
Since $|a-b|<d_i$ and $d_i$ is coprime with $m_i/c_i$, the left hand side is not divisible by $d_i$, which gives the required contradiction. By the same token, \eqref{loop} is the shortest path connecting $(v_i,r)$ with itself. Any longer path would need to pass through the same vertex at least twice, hence it would not be admissible. This proves the uniqueness. Thus, If $i=j$, $n_{ii}^{rs} = \delta_{rs}$, i.e.\ there is a single loop attached to each vertex in $L_{2n+1}^{N;\m}$ and there are no links between vertices $v_i^r$ and $v_i^s$ if $r\neq s$. 

If all the $m_i$ are coprime with $N$, the graph $L_{2n+1}^{N;\m}$ coincides with the graph described in \cite[p.\ 257--258]{HonSzy:len}. At the other extreme, i.e.\ if all the $m_i$ divide $N$, then the graph $L_{2n+1}^{N;\m}$ consists of $n+1$ levels of interconnected loops with $m_i$ mutually disconnected loops at the $i$-th level.

\begin{example}
(1) $L_{3}^{kl; 1,l}$ consists of one loop at level 0 and $l$-loops at level 1 with $k$ links connecting the loop in  level 0 with each of the loops at  level 1, so the corresponding graph is:

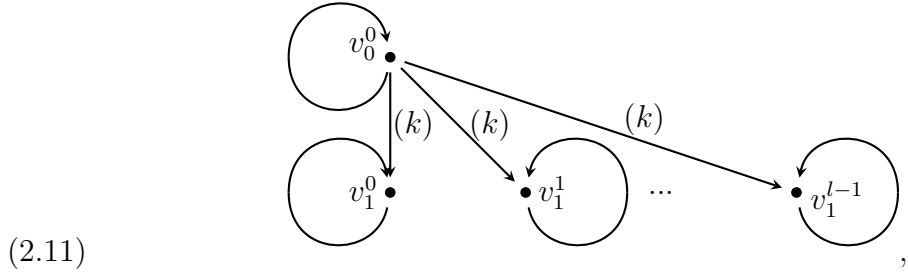
\begin{equation}
\begin{tikzpicture}[scale=1.8]

\node (0_0) at (0,0) [circle] {};
\node (1_0) at (1,0) [circle] {};
\node (2_0) at (2,0) [circle] {...};
\node (3_0) at (3,0) [circle] {};
\node (0_1) at (0,1) [circle] {};
\node (1_1) at (1,1) [circle] {};
\node (3_1) at (3,1) [circle] {};

\foreach \y in {0,1} \filldraw  (0_\y) circle (1pt);
\foreach \y in {0}  \filldraw  (1_\y) circle (1pt);
\foreach \y in {0}  \filldraw (3_\y) circle (1pt);

\draw (-0.2,1.1) node {$v_0^0$};
\draw (-0.2,0) node {$v_1^0$};
\draw (1.2,0) node {$v_1^1$};
\draw (3.3,-.05) node {$v_1^{l-1}$};

\draw[-stealth,thick] (0_1)-- (0_0) node[pos=0.5, inner sep=0.5pt, anchor=west] {$(k)$};
\draw[-stealth,thick] (0_1)-- (1_0) node[pos=0.5, inner sep=0.5pt, anchor=west] {$\;(k)$};
\draw[-stealth,thick] (0_1)-- (3_0) node[pos=0.6, inner sep=0.5pt, anchor=south] {$\;\;\;(k)$};

\draw[-stealth,thick] (-.75,0) .. controls (-.75,.5) and (-.1,.5) .. (0_0);
\draw[thick] (-.75,0) .. controls (-.75,-.5) and (-.1,-.5) .. (0_0);

\draw[-stealth,thick] (3.75,0) .. controls (3.75,.5) and (3.1,.5) .. (3_0); 
\draw[thick] (3.75,0) .. controls (3.75,-.5) and (3.1,-.5) .. (3_0);

\draw[-stealth,thick] (1.75,0) .. controls (1.75,.5) and (1.1,.5) .. (1_0); 
\draw[thick] (1.75,0) .. controls (1.75,-.5) and (1.1,-.5) .. (1_0);

\draw[-stealth,thick] (-.75,1) .. controls (-.75,1.5) and (-.1,1.5) .. (0_1);
\draw[thick] (-.75,1) .. controls (-.75,.5) and (-.1,.5) .. (0_1);

\end{tikzpicture},
\end{equation}
where the labels in brackets over the straight arrows indicate their multiplicities.

(2) $L_{5}^{kl; 1,1,l}$ consists of one loop at level 0, one loop at level 1 and $l$-loops at level 2 with the numbers of edges connecting different levels given by 
$$
n^{00}_{01} = kl,\qquad  n^{0r}_{02}= \frac{kl(k+1)}{2} - rk, \qquad n^{0r}_{12} = k, 
$$
hence the corresponding graph comes out as:
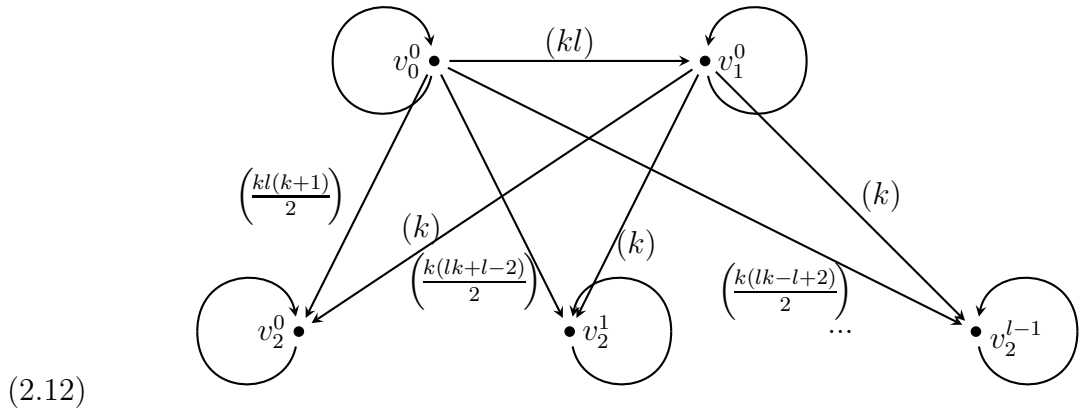
\begin{equation}
\begin{tikzpicture}[scale=1.8]

\node (0_0) at (0,0) [circle] {};
\node (1_0) at (2,0) [circle] {};
\node (3_0) at (4,0) [circle] {...};
\node (4_0) at (5,0) [circle] {};
\node (0_1) at (0,2) [circle] {};
\node (1_1) at (1,2) [circle] {};
\node (3_1) at (3,2) [circle] {};

\foreach \y in {0,1} \filldraw  (1_\y) circle (1pt);
\foreach \y in {0}  \filldraw  (0_\y) circle (1pt);
\foreach \y in {0}  \filldraw (4_\y) circle (1pt);
\foreach \y in {1} \filldraw  (3_\y) circle (1pt);

\draw (.8,2) node {$v_0^0$};
\draw (3.2,2) node {$v_1^0$};
\draw (-0.2,0) node {$v_2^0$};
\draw (2.2,0) node {$v_2^1$};
\draw (5.3,-.05) node {$v_2^{l-1}$};

\draw[-stealth,thick] (3_1)-- (0_0) node[pos=0.7, inner sep=0.5pt, anchor=south] {$(k)\;$};
\draw[-stealth,thick] (3_1)-- (1_0) node[pos=0.7, inner sep=0.5pt, anchor=west] {$(k)$};
\draw[-stealth,thick] (3_1)-- (4_0) node[pos=0.5, inner sep=0.5pt, anchor=west] {$\;\;(k)$};
\draw[-stealth,thick] (1_1)-- (3_1) node[pos=0.5, inner sep=0.5pt, anchor=south] {$(kl)$};
\draw[-stealth,thick] (1_1)-- (0_0) node[pos=0.5, inner sep=0.5pt, anchor=east] {$\left(\!\frac{kl(k+1)}{2}\!\right)\;\;$}; 
\draw[-stealth,thick] (1_1)-- (1_0) node[pos=0.85, inner sep=0.5pt, anchor=east] {$\left(\!\frac{k(lk+l-2)}{2}\!\right)$}; 
\draw[-stealth,thick] (1_1)-- (4_0) node[pos=0.75, inner sep=0.5pt, anchor=north] {$\left(\!\frac{k(lk-l+2)}{2}\!\right)\;\;\;\;\;\;\;\;\;\;\;$}; 

\draw[-stealth,thick] (-.75,0) .. controls (-.75,.5) and (-.1,.5) .. (0_0);
\draw[thick] (-.75,0) .. controls (-.75,-.5) and (-.1,-.5) .. (0_0);

\draw[-stealth,thick] (5.75,0) .. controls (5.75,.5) and (5.1,.5) .. (4_0); 
\draw[thick] (5.75,0) .. controls (5.75,-.5) and (5.1,-.5) .. (4_0);

\draw[-stealth,thick] (2.75,0) .. controls (2.75,.5) and (2.1,.5) .. (1_0); 
\draw[thick] (2.75,0) .. controls (2.75,-.5) and (2.1,-.5) .. (1_0);

\draw[-stealth,thick] (.25,2) .. controls (.25,2.5) and (.9,2.5) .. (1_1);
\draw[thick] (.25,2) .. controls (.25,1.5) and (.9,1.5) .. (1_1);

\draw[-stealth,thick] (3.75,2) .. controls (3.75,2.5) and (3.1,2.5) .. (3_1);
\draw[thick] (3.75,2) .. controls (3.75,1.5) and (3.1,1.5) .. (3_1);

\end{tikzpicture}
\end{equation}

(3) $L_{5}^{kl; 1,l,l}$ consists of one loop at level 0, $l$ loops at level 1 and $l$ loops at level 2 with the numbers of edges connecting different levels given by 
$$
n^{0r}_{01} = k,\qquad  n^{0r}_{02}= \frac{k(k+1)}{2}, \qquad n^{0r}_{12} = k,
$$
i.e.\
\begin{equation}
\begin{tikzpicture}[scale=1.8]

\node (0_0) at (0,0) [circle] {};
\node (1_0) at (3,0) [circle] {};
\node (2_0) at (4,0) [circle] {...};
\node (3_0) at (6,0) [circle] {};
\node (0_1) at (0,1) [circle] {};
\node (1_1) at (3,1) [circle] {};
\node (2_1) at (4,1) [circle] {...};
\node (3_1) at (6,1) [circle] {};
\node (0_2) at (4,2) [circle] {};

\foreach \y in {0,1} \filldraw  (0_\y) circle (1pt);
\foreach \y in {0,1} \filldraw  (1_\y) circle (1pt);
\foreach \y in {0,1} \filldraw  (3_\y) circle (1pt);
\foreach \y in {2} \filldraw  (0_\y) circle (1pt);

\draw (4,2.2) node {$v_0^0$};
\draw (-.2,1) node {$v_1^0$};
\draw (2.8,1) node {$v_1^1$};
\draw (6.3,1) node {$v_1^{l-1}$};
\draw (-0.2,0) node {$v_2^0$};
\draw (2.8,0) node {$v_2^1$};
\draw (6.3,0) node {$v_2^{l-1}$};

\foreach \x in {0,3} \draw[-stealth,thick] (0_2)-- (\x_1) node[pos=0.7, inner sep=0.5pt, anchor=south] {$\;\;\;(k)$};
\foreach \x in {1} \draw[-stealth,thick] (0_2)-- (\x_1) node[pos=0.7, inner sep=0.5pt, anchor=south] {$(k)\;\;$};
\foreach \x in {1} \draw[-stealth,thick] (0_2)-- (\x_0) node[pos=0.7, inner sep=0.5pt, anchor=west] {$\;\left(\frac{k(k+1)}{2}\right)$};
\foreach \x in {3} \draw[-stealth,thick] (0_2)-- (\x_0) node[pos=0.75, inner sep=0.5pt, anchor=east] {$\left(\frac{k(k+1)}{2}\right)\;\;$};
\foreach \x in {0} \draw[-stealth,thick] (0_2)-- (\x_0) node[pos=0.75, inner sep=0.5pt, anchor=west] {$\;\;\;\;\left(\frac{k(k+1)}{2}\right)$};
\draw[-stealth,thick] (0_1)-- (0_0) node[pos=0.5, inner sep=0.5pt, anchor=west] {$(k)$};
\draw[-stealth,thick] (1_1)-- (1_0) node[pos=0.5, inner sep=0.5pt, anchor=east] {$(k)$};
\draw[-stealth,thick] (3_1)-- (3_0) node[pos=0.5, inner sep=0.5pt, anchor=west] {$(k)$};

\draw[-stealth,thick] (-.75,1) .. controls (-.75,1.5) and (-.1,1.5) .. (0_1);
\draw[thick] (-.75,1) .. controls (-.75,.5) and (-.1,.5) .. (0_1);

\draw[-stealth,thick] (-.75,0) .. controls (-.75,.5) and (-.1,.5) .. (0_0);
\draw[thick] (-.75,0) .. controls (-.75,-.5) and (-.1,-.5) .. (0_0);

\draw[-stealth,thick] (6.75,1) .. controls (6.75,1.5) and (6.1,1.5) .. (3_1); 
\draw[thick] (6.75,1) .. controls (6.75,.5) and (6.1,.5) .. (3_1);

\draw[-stealth,thick] (6.75,0) .. controls (6.75,.5) and (6.1,.5) .. (3_0); 
\draw[thick] (6.75,0) .. controls (6.75,-.5) and (6.1,-.5) .. (3_0);

\draw[-stealth,thick] (2.25,1) .. controls (2.25,1.5) and (2.9,1.5) .. (1_1); 
\draw[thick] (2.25,1) .. controls (2.25,.5) and (2.9,.5) .. (1_1);

\draw[-stealth,thick] (2.25,0) .. controls (2.25,.5) and (2.9,.5) .. (1_0); 
\draw[thick] (2.25,0) .. controls (2.25,-.5) and (2.9,-.5) .. (1_0);

\draw[-stealth,thick] (4,2.75) .. controls (4.5,2.75) and (4.5,2.1) .. (0_2);
\draw[thick] (4,2.75) .. controls (3.5,2.75) and (3.5,2.1) .. (0_2);

\end{tikzpicture}
\end{equation}

(4) $L_{7}^{kl; 1,1,1,l}$ consists of one loop each at levels 0, 1 and 2, and $l$-loops at level 3 with the numbers of edges connecting different levels given by 
$$
n^{00}_{01} = n^{00}_{12} = kl,\qquad  n^{0r}_{13}= \frac{kl(k+1)}{2} - rk, \qquad n^{0r}_{23} = k, \qquad  n^{00}_{02}= \frac{kl(kl+1)}{2},
$$
$$
n^{0r}_{03}= \frac{kl(k+1)}{12} (2kl +l+3) + \frac{kr(r-1)}{2} - \frac{kl(k+1)}{2} r.
$$
\end{example}

The main result of this section is contained in the following
\begin{theorem}\label{thm.lens.main}
As $C^*$-algebras,
$$
C^*(L_{2n+1}^{N;\m}) \cong C(L_q^{2n+1}(N; \m)).
$$
\end{theorem}
\begin{proof}
Since $C(L_q^{2n+1}(N;\m))$ is obtained as fixed points of a finite abelian group action on a graph $C^*$-algebra,  \cite[Theorem~4.6]{Cri:cor} implies that it is isomorphic to 
$$
\left( \sum_{i=0}^n P_{(v_i,0)} \right) C^*(L_{2n+1}\times_c \ZZ_N)\left( \sum_{i=0}^n P_{(v_i,0)} \right).
$$
Thus it suffices to prove that the following map
$$
\psi: C^*(L_{2n+1}^{N;\m})  \to \left( \sum_{i=0}^n P_{(v_i,0)} \right)C^*(L_{2n+1}\times_c \ZZ_N)\left( \sum_{i=0}^n P_{(v_i,0)} \right),
$$
given by
$$
P_{v_i^r} \mapsto P_{(v_i,r)}, \qquad i=0,\ldots, n,\quad r=0,1,\ldots, c_i-1,
$$
and, for all admissible paths 
$$
\alpha = {(e_{ii_1}, r+m_i)(e_{i_1i_2}, r_1)\ldots (e_{{i_k}i_{k+1}}, r_k)(e_{i_{k+1}j},s)},
$$
$$
S_\alpha \mapsto S_{(e_{ii_1}, r+m_i)}S_{(e_{i_1i_2}, r_1)}\cdots S_{(e_{{i_k}i_{k+1}}, r_k)}S_{(e_{i_{k+1}j},s)},
$$
extends to a $C^*$-algebra isomorphism.

In view of the universal property of the graph $C^*$-algebra, to prove that $\psi$ extends to a 
$*$-homomorphism it suffices to check that the images of the $P_{v_i^r}$ and $S_\alpha$ under $\psi$ satisfy relations \eqref{graph} for $L_{2n+1}^{N;\m}$. Conditions \eqref{ss} are obvious. To prove \eqref{p}, we fix $(v_i,r)\in L_{2n+1}\times_c \ZZ_N$, and, for any $\nu\in \NN$, split the set of all admissible paths from $(v_i,r)$  into the subsets $A_\nu$ of those of length less than $n$ and $B_\nu$ of those of length exactly $\nu$. We will prove by induction on $\nu$ that
\begin{equation}\label{split}
P_{(v_i,r)} = \sum_{\alpha \in A_\nu} S_\alpha S^*_\alpha + \sum_{\beta \in B_\nu} S_\beta S^*_\beta .
\end{equation}
The equation \eqref{split} obviously holds if $\nu=1$. Now, suppose that it also holds for some (other) $\nu$, and let 
$$
\beta = {(e_{ii_1}, r+m_i)(e_{i_1i_2}, r_1)\ldots (e_{{i_k}i_{k+1}}, r_k)(e_{i_{k+1}j},s)} \in B_\nu.
$$
Using \eqref{p} at the range vertex $(v_j,s)$ of $\beta$ one easily finds that
\begin{equation}\label{s+1}
S_\beta S_\beta^* = \sum_{l= s}^nS_\beta S_{(e_{jl}, s+m_j)}S^*_{(e_{jl}, s+m_j)} S_\beta^*.
\end{equation}
All the paths 
$$
\beta' = {(e_{ii_1}, r+m_i)(e_{i_1i_2}, r_1)\ldots (e_{{i_k}i_{k+1}}, r_k)(e_{i_{k+1}j},s)}(e_{jl}, s+m_j) 
$$
that appear in \eqref{s+1} are admissible, which is obvious in case $l\neq s$. Otherwise, this follows
by the observation that no segment of an admissible path that is not a loop can have length greater than $d_i -1$ at any given level $i$ (otherwise the path would pass through the same vertex at least twice); if an edge with both the source and range at the level $i$ is added its range will be a vertex that is not a range for any edge yet, otherwise one is led to contradiction as in \eqref{contra}.  Therefore, $\beta' \in A_{\nu +1}\setminus A_\nu$ or $\beta'\in B_{\nu+1}$ and, using the inductive hypothesis, we obtain
\begin{eqnarray*}
 P_{(v_i,r)} &=& \sum_{\alpha \in A_\nu} S_\alpha S^*_\alpha + \sum_{\beta \in B_\nu} S_\beta S^*_\beta \\
 &=& \sum_{\alpha \in A_\nu} S_\alpha S^*_\alpha + \sum_{\beta' \in B_{\nu+1}} S_{\beta'} S^*_{\beta'}+ \sum_{\beta' \in A_{\nu +1}\setminus A_\nu} S_{\beta'} S^*_{\beta'} \\
&=& \sum_{\alpha \in A_{\nu+1}} S_\alpha S^*_\alpha + \sum_{\beta \in B_{\nu+1}} S_\beta S^*_\beta .
\end{eqnarray*}
By the principle of mathematical induction, \eqref{split} holds for all natural $\nu$.
 
Since $L_{2n+1}\times_c \ZZ_N$ is a finite graph there exists $\nu$ such that $B_\nu$ is an empty set, and hence relation \eqref{p} for $P_{(v_i,r)} = \psi (P_{v_i^r})$ is equivalent to \eqref{split} for this $\nu$. This proves that $\psi$ extends to a $*$-homomorphism, which we still denote $\psi$. To prove that $\psi$ is injective, we apply 
the general Cuntz-Krieger uniqueness theorem, \cite[Theorem 1.2]{Szy:ck}. Clearly, $\psi(P_{(v_i,r)}) \neq 0$ 
for all vertices $v_i^r$. Also, the only loops without exits in graph $L_{2n+1}^{N;\m}$ are the edges 
$e_n^r:= e_{nn;1}^{rr}$ attached to vertices at level $n$. For all $r=0,\ldots c_n-1$, we have 
$$
\psi(S_{e^r_n}) = (e_{nn}, r+m_n)(e_{nn},r+2m_n) \ldots (e_{nn}, r+(d_n-1)m_n)(e_{nn}, r). 
$$
It follows from the last part of the proof of Theorem 2.4 in \cite{KPR:gralg} that this 
is a partial unitary with full spectrum. Thus the hypothesis of \cite[Theorem 1.2]{Szy:ck} holds and $\psi$ 
is injective. 

Finally, we need to prove that $\psi$ is surjective. First, let us consider a path $\alpha$ with source and range both in the set $\{(v_i, 0) \; |\; i=0,\ldots ,n\}$. Since every loop in the crossed-product graph $L_{2n+1}\times_c {\ZZ_N}$ passes through one of the vertices $(v_i ,r_i)$, $r_i=0,\ldots c_i-1$, $i = 0,\ldots ,n$, the path $\alpha$ is a concatenation of admissible paths, i.e.\ $\alpha = \alpha_1\alpha_2\ldots \alpha_k$, with all the $\alpha_k$ admissible. Therefore $S_\alpha = S_{\alpha_1}S_{\alpha_2}\cdots S_{\alpha_k}$ is in the image of 
$C^*(L_{2n+1}^{N;\m})$ under $\psi$.
\end{proof}

\begin{corollary}\label{cor.ex.seq.lens}
The following sequence of $C^*$-algebras
\begin{equation}\label{ses.lens}
\xymatrix{  0 \ar[r] & (\cK \otimes C(\TT))^{\oplus \gcd(m_n, N)} \ar[r] & C(L_q^{2n+1}(N; \m)) \ar[r] & C(L_q^{2n-1}(N; \m)) \ar[r] & 0,}
\end{equation}
where $\cK$ denotes compact operators on a separable Hilbert space, is exact.
\end{corollary}
\begin{proof}
Each level $n$ vertex in graph $L_{2n+1}^{N;\m}$ emits exactly one edge, to itself. On the other hand, 
there exist infinitely many paths from vertices in other levels to each vertex in level $n$. Thus the closed 
two-sided ideal of $C^*(L_{2n+1}^{N;\m})$ generated by projections $P_{v_n^i}$, $i=0,\ldots, \gcd(m_n, N)-1$, 
is isomorphic to $(\cK \otimes C(\TT))^{\oplus \gcd(m_n, N)}$, \cite{KPR:gralg} and \cite{BPRS:rfgralg}, 
and the corresponding quotient is isomorphic to $C^*(L_{2n-1}^{N;\m})$, 
\cite{BPRS:rfgralg}. Thus the exactness of (\ref{ses.lens}) follows from Theorem \ref{thm.lens.main}. 
\end{proof} 

The identification of the algebra of continuous functions on the quantum lens space with the graph $C^*$-algebra $C^*(L_{2n+1}^{N;\m})$ allows one to design a method for computing $K$-groups of $C(L_q^{2n+1}(N; \m))$. More precisely,  
$$
K_0(C(L^{2n+1}_q(N;\m))) = \coker\, \Phi, \qquad K_1(C(L^{2n+1}_q(N;\m))) = \ker \Phi, 
$$
where $\Phi$ is the endomorphism of a free abelian group with generators $v_i^r$ given by
$$
\Phi(v_i^r) = \sum_{j,s} (n_{ij}^{rs} - \delta_{ij}\delta_{rs})v_j^s;
$$
see \cite[Theorem~3.2]{RaeSzy:Cun}. The complete computation of $K_1(C(L^{2n+1}_q(N;\m)))$ is presented in \cite[Proposition~5.2]{BrzFai:not}, the more difficult computation of $K_0(C(L^{2n+1}_q(N;\m)))$ boils down to detailed analysis of numbers of links connecting various loops 
and then to derive the Smith normal form of the matrix corresponding to the transformation $\Phi$. Recall \cite{Smi:sys} that every integer matrix $A$ can be reduced (by row and column operations) to the diagonal form with entries 0 or $\alpha_1,\ldots, \alpha_n$, where 
\begin{equation}\label{Smith}
\alpha_1 = \Delta_1, \qquad \alpha_{i+1} = \frac{\Delta_{i+1}}{\Delta_i},
\end{equation}
where the $\Delta_i$ are greatest common divisiors of all minors of $A$ of size $i$. The torsion part of the cokernel of the transformation defined by $A$ is then
$$
\ZZ_{\alpha_1}\oplus \ZZ_{\alpha_2}\oplus \ldots \oplus \ZZ_{\alpha_n}.
$$
 Below we give three examples of this.

\begin{example}\label{ex.k5}
$$
K_0(C(L_q^5(kl;1,1,l))) = \ZZ^l \oplus  \begin{cases} \ZZ_k \oplus \ZZ_k & \mbox{if $k$ is odd or $l$ is even},\cr \ZZ_{2k}\oplus \ZZ_{\frac{k}2} & \mbox{if $k$ is even and $l$ is odd}.
\end{cases}
$$
\end{example}
\begin{proof}
The transformation $\Phi$ is determined by the integer $(l+2)\times (l+2)$-matrix with  the first two columns 
$$
\begin{pmatrix}
0 & 0 \cr  kl & 0\cr  kl(k+1)/2 & k \cr kl(k+1)/2 - k & k\cr kl(k+1)/2 - 2k & k\cr \ldots & \ldots \cr kl(k+1)/2 - k(l-1) & k
\end{pmatrix},
$$
and all other entries 0. By simple row and column operations this matrix can be reduced to the block diagonal form
$$
\begin{pmatrix} 
k & 0 & 0\cr kl(k-1)/2 & k & 0 \cr 0 & 0 &0
\end{pmatrix}. 
$$
The right bottom corner gives the infinite part of $K_0(C(L_q^5(kl;1,1,l)))$. The 
greatest common divisor of the minor of size 2 is $\Delta_2 =k^2$, while the 
greatest common divisor of one-dimensional minors  depends on the parity of $k$ and $l$,
$$
\Delta_1= \begin{cases} k  & \mbox{if $k$ is odd or $l$ is even},\cr {\frac{k}2} & \mbox{if $k$ is even and $l$ is odd}
\end{cases}.
$$
In view of \eqref{Smith}, this yields the stated finite part of $K_0(C(L_q^5(kl;1,1,l)))$.
\end{proof}

\begin{example}\label{ex.k5'}
$$
K_0(C(L_q^5(kl;1,l,l))) = \ZZ^l \oplus  \begin{cases} \ZZ_k^{l+1}  & \mbox{if $k$ is odd},\cr \ZZ_{2k}\oplus \ZZ_{\frac{k}2}\oplus \ZZ_k^{l-1} & \mbox{if $k$ is even}.
\end{cases}
$$
\end{example}
\begin{proof}
The transformation $\Phi$ is determined by the integer $(2l+1)\times (2l+1)$-matrix with  the first $l+1$ columns 
$$
\begin{pmatrix}
0 & 0 & 0 &\ldots & 0 \cr 
k & 0 & 0 &\ldots & 0 \cr  
k & 0 & 0 &\ldots & 0 \cr  
 \ldots & \ldots & \ldots &\ldots & 0 \cr 
 k & 0 & 0 &\ldots & 0 \cr   
 k(k+1)/2 & k & 0 &\ldots & 0 \cr
 k(k+1)/2 & 0 & k &\ldots & 0 \cr
  \ldots & \ldots & \ldots &\ldots & 0 \cr
 k(k+1)/2 & 0 & 0 &\ldots & k 
 \end{pmatrix},
$$
and all other entries 0. By subtracting the $l+1$-st row from rows $2$ to $l$, the first $l$ rows can be reduced to the zeros. This gives the infinite part of the group $K_0(C(L_q^5(kl;1,l,l)))$. The greatest common divisors of the minors in the remaining matrix come out as
$$
\Delta_i = \begin{cases} k^i  & \mbox{if $k$ is odd},
\cr {\frac{k^i}2} & \mbox{if $k$ is even}
\end{cases}, \qquad i=1,2,\ldots , l,
$$
and $\Delta_{l+1} = k^{l+1}$. 
In view of \eqref{Smith}, this yields the  finite part of $K_0(C(L_q^5(kl;1,l,l)))$ as stated.
\end{proof}

\begin{example}\label{ex.k7}
Let
\begin{equation}\label{alpha.beta}
\alpha :=n^{00}_{13} = \frac{kl(k+1)}{2}, \qquad \beta := n^{00}_{03} =\alpha\frac{2kl +l+3}{6}.
\end{equation} 
Then
\begin{equation}\label{k7}
K_0(C(L_q^7(kl;1,1,1,l))) = \ZZ^l \oplus  \begin{cases} \ZZ_k \oplus \ZZ_k \oplus \ZZ_k & \mbox{if $k | \alpha$ \& $k | \beta$},\cr
\ZZ_{\frac k6} \oplus \ZZ_k \oplus \ZZ_{6k} & \mbox{if $k |\alpha$ \& $\beta \equiv \frac k6 , \frac{5k}6$ (mod $k$)},\cr
\ZZ_{\frac k3} \oplus \ZZ_k \oplus \ZZ_{3k} & \mbox{if $k | \alpha$ \& $\beta \equiv \frac k3 , \frac{2k}3$ (mod $k$)},\cr
\ZZ_{\frac k2} \oplus \ZZ_k \oplus \ZZ_{2k} & \mbox{if $k | \alpha$ \& $\beta \equiv \frac k2 $ (mod $k$)}, \cr
\ZZ_{\frac k2} \oplus \ZZ_{\frac k2} \oplus \ZZ_{4k} & \mbox{if $k \not{\mid}\  \alpha$ \& $\frac k2  | \beta$},\cr
\ZZ_{\frac k6} \oplus \ZZ_{\frac k2} \oplus \ZZ_{12k} & \mbox{if $k \not |\ \alpha$ \& $\frac k2  \not |\ \beta$}.
\end{cases}
\end{equation}
\end{example}
\begin{proof}
Let us first observe that 
$$
n^{0r}_{13} = \alpha - kr, \qquad n^{0r}_{03} = \beta - \alpha r + \frac{r(r-1)}{2} k.
$$
Therefore, the matrix representing $\Phi$ is an $(l+3) \times (l+3)$-matrix with the non-zero entries contained in the first three columns
$$
\begin{pmatrix}
0 & 0 & 0 \cr
kl & 0 & 0 \cr
\frac{kl(kl+1)}{2} & kl & 0\cr
\beta & \alpha & k \cr
\beta - \alpha & \alpha -k & k\cr
\beta - 2\alpha + k & \alpha -2k & k\cr
\ldots & \ldots & \ldots \cr
\beta - (l-1)\alpha + \frac{(l-1)(l-2)}{2}k & \alpha - (l-1)k & k
\end{pmatrix} .
$$
By elementary row and column operations (starting with subtracting row 4 from all subsequent rows) and using the divisibility of the entries by $k$, we arrive at the block diagonal matrix
$$
\begin{pmatrix}
k & 0 & 0 & 0\cr
\alpha & k & 0 & 0\cr
\beta & \alpha & k & 0\cr
0 & 0 & 0 & 0
\end{pmatrix}.
$$
The zero $l\times l$-matrix in the bottom-right corner gives $\ZZ^l$ as the infinite part of the group $K_0(C(L_q^7(kl;1,1,1,l)))$, while the $3\times 3$-matrix in the top-left corner gives the finite part of $K_0(C(L_q^7(kl;1,1,1,l)))$. Its nature depends on the divisbiliity properties of $\alpha$ and $\beta$ and it splits into two parts. If  $\alpha$ is divisible by $k$, then the matrix can be reduced to 
$$
\begin{pmatrix}
k & 0 & 0 \cr
 0 & k & 0 \cr
\beta & 0 & k 
\end{pmatrix}.
$$
The greatest common divisors of the minors thus read
$$
\Delta_1 = \gcd(k,\beta),\qquad \Delta_2 = \gcd(k^2,k\beta) = k \Delta_1, \qquad \Delta_3 =k^3,
$$
thus yielding the diagonal entries:
$$
\alpha_1 = \gcd(k,\beta), \qquad \alpha_2 = k, \qquad \alpha_3 = \frac{k^2}{\gcd(k,\beta)}.
$$
If $k$ does not divide $\alpha$, then it must be even, and $\alpha$ is divisible by $k/2$, thus leading to the matrix
$$
\begin{pmatrix}
k & 0 & 0 \cr
 \frac{k}{2} & k & 0 \cr
\beta & \frac k2 & k 
\end{pmatrix},
$$
and the corresponding Smith normal form entries
$$
\alpha_1 = \gcd(\frac{k}2,\beta), \qquad \alpha_2 = \frac{k}2, \qquad \alpha_3 = \frac{2k^2}{\gcd(\frac{k}2,\beta)}.
$$
Let us note that
$$
\beta = \frac{k(k+1)(k+2)l(l+1)}{12} +  \frac{(k-1)k(k+1)l(l-1)}{12}.
$$
Hence $\beta$ modulo $k$ has to be a multiple of the sixth of $k$, and, in the case of even $k$, $\beta$ modulo $k/2$ has to be a multiple of the third of $k/2$. The analysis of all these possibilities yields the stated form of $K_0(C(L_q^7(kl;1,1,1,l)))$.
\end{proof}

\begin{remark}\label{rem.comparison}
If $l=1$, the results of Example~\ref{ex.k5} agree with that of \cite[Proposition~2.3]{HonSzy:len}. On the other hand, if $l=1$, then numbers $\alpha$ and $\beta$ defined in \eqref{alpha.beta} come out as 
$$
\alpha = \frac{k(k+1)}{2}, \qquad \beta =\alpha\frac{k+2}{3},
$$
and the second and fourth cases in \eqref{k7} cannot occur. The remaining cases coincide with the $K$-groups computed in \cite[Example~6.6]{AriBra:Gys}.
\end{remark}

\section{$K$-theory of quantum weighted projective  spaces}\label{sec.proj}\setcounter{equation}{0}

The aim of this section is to calculate $K$-theory of a fairly general class of quantum weighted projective spaces and to give a complete description of $C^*$-algebras of continuous functions on quantum weighted projective lines as graph AF-algebras.

As before, we fix a sequence of  positive integers $\m := m_0,\ldots, m_n$. In addition to the $\ZZ_N$-action \eqref{action}, the algebra $C(S^{2n+1}_q)$ admits the circle group action $\varrho_{\m}$, 
\begin{equation}\label{action.circle}
\varrho_{\m}: z_i \mapsto \xi^{m_i} z_i, \qquad i=0,\ldots , n,
\end{equation}
where $\xi$ is the unitary generator of $\TT$ (of infinite order). Fixed points $C(\WW\PP_q^n(\m))$ form the algebra of continuous functions on the quantum weighted projective space, \cite{BrzFai:tea}. As explained in \cite{BrzFai:not}, for a fixed $N$, all the elements $\sum_i x_i$ of  $C(S^{2n+1}_q)$ that transform according to the rule 
$$
\sum_ix_i\mapsto \sum_i\xi^{r_iN} x_i,  \qquad r_i\in \ZZ,
$$
form a subalgebra of $C(S^{2n+1}_q)$ isomorphic to $C(L^{2n+1}_q(N; \m))$. The action $\varrho_{\m}$  gives rise to the $\TT$-action $\hat\varrho_{\m}$ on $C(L^{2n+1}_q(N; \m))$ with fixed points being again $C(\WW\PP_q^n(\m))$: an element $x \in C(L^{2n+1}_q(N; \m))$ transforms under $\hat\varrho_{\m}$ as $x\mapsto \xi^{r} x$ provided it transforms as $x\mapsto \xi^{rN} x$ under $\varrho_{\m}$.  The actions $\varrho^{N}_\m$ and  $\hat\varrho_{\m}$ can be uderstood as being derived from $\varrho_\m$ via the short exact sequence of abelian groups
$$
\xymatrix{ 1 \ar[r] & \TT \ar[r] &\TT \ar[r] & \ZZ_N \ar[r] & 1,
}
$$
where the  (non-trivial) monomorphism is  $\xi \mapsto \xi^N$ and the (non-trivial) epimorphism is $\xi\mapsto\zeta$; see e.g.\ \cite[Section~A.1.1]{NasVan:gra}.

We describe $C^*$-algebras of the quantum weighted projective spaces $\WW\PP_q^1(\m)$, 
$\m=(m_0,m_1)$, as AF
graph algebras. Let  $g:=\gcd(m_0,m_1)$, $\tilde{m}_0:=m_0/g$ and $\tilde{m_1}:=m_1/g$, and define a graph 
$W_1(\m)$ as follows. The graph has $\tilde{m}_1 +1$ vertices, denoted $w_0,\ldots,w_{\tilde{m}_1}$. 
For each $j\in\{1,\ldots,w_{\tilde{m}_1}\}$ there are infinitely many edges from $w_0$ to $w_j$, denoted 
$f_{jk}$, $k\in\NN$, i.e.
\begin{equation}
\begin{tikzpicture}[scale=1.8]

\node (0_0) at (0,0) [circle] {};
\node (1_0) at (1,0) [circle] {};
\node (2_0) at (2,0) [circle] {...};
\node (3_0) at (3,0) [circle] {};
\node (0_1) at (0,1) [circle] {};
\node (1_1) at (1,1) [circle] {};
\node (3_1) at (3,1) [circle] {};

\foreach \y in {0,1} \filldraw  (0_\y) circle (1pt);
\foreach \y in {0}  \filldraw  (1_\y) circle (1pt);
\foreach \y in {0}  \filldraw (3_\y) circle (1pt);

\draw (0.2,1.1) node {$w_0$};
\draw (-0.2,0) node {$w_1$};
\draw (0.8,0) node {$w_2$};
\draw (3.3,-.05) node {$w_{\tilde{m}_1}$};

\draw[-stealth,thick] (0_1)-- (0_0) node[pos=0.5, inner sep=0.5pt, anchor=west] {$(\infty)$};
\draw[-stealth,thick] (0_1)-- (1_0) node[pos=0.5, inner sep=0.5pt, anchor=west] {$\;(\infty)$};
\draw[-stealth,thick] (0_1)-- (3_0) node[pos=0.6, inner sep=0.5pt, anchor=south] {$\;\;\;(\infty)$};

\end{tikzpicture}
\end{equation}

\begin{proposition}\label{3dimweighted}
For all values of $\m = (m_0,m_1)$, $C(\WW\PP_q^1(\m))$ is  an AF-algebra isomorphic to the graph $C^*$-algebra $C^*(W_1(\m))$. Consequently,
$$
K_0(C(\WW\PP_q^1(\m))) = \ZZ^{1+m_1/\gcd(m_0,m_1)}, \qquad K_1(C(\WW\PP_q^1(\m))) = 0.
$$
\end{proposition}
\begin{proof}
As explained above, $C(\WW\PP_q^1(\m))$ is isomorphic to the $C^*$-algebra of fixed 
points for the generalized gauge action of the circle group $\TT$ on the graph algebra $C^*(L_3)$, such that 
$$ \varrho_{\m}(S_{e_{ij}}) = \xi^{\m_i}S_{e_{ij}} \;\;\; \text{for} \; i=0,1, \; j=i,1. $$ 
We denote this fixed point algebra $C^*(L_3)^{\varrho_{\m}}$ 
and  we construct a $C^*$-algebra isomorphism $\phi:C^*(W_1(\m)) \to C^*(L_3)^{\varrho_{\m}}$. 
At first we find targets for the generators of  $C^*(W_1(\m))$ inside $C^*(L_3)$. Let 
$$ \begin{aligned}
\phi(P_{w_0}) & := P_{v_0} - \sum_{j=2}^{\tilde{m}_1} S_{e_{00}}^{j-2} S_{e_{01}}S_{e_{01}}^*
   (S_{e_{00}}^*)^{j-2}, \\
\phi(P_{w_1}) & := P_{v_1}, \\
\phi(P_{w_j}) & := S_{e_{00}}^{j-2} S_{e_{01}}S_{e_{01}}^*(S_{e_{00}}^*)^{j-2}, \;\;\; \text{for} \; 
   j=2,\dots,\tilde{m}_1, \\ 
\phi(S_{f_{1k}}) & := S_{e_{00}}^{\tilde{m}_1(k+1)-1} S_{e_{01}} (S_{e_{11}}^*)^{\tilde{m}_0(k+1)}, 
  \;\;\; \text{for} \; k\in\NN, \\
\phi(S_{f_{jk}}) & := S_{e_{00}}^{\tilde{m}_1(k+1)+j-2} S_{e_{01}} (S_{e_{11}}^*)^{\tilde{m}_0(k+1)}
   S_{e_{01}}^* (S_{e_{00}}^*)^{j-2},  \;\;\; \text{for} \; k\in\NN, \; j=2,\dots,\tilde{m}_1.  
\end{aligned} $$
Clearly, these elements of $C^*(L_3)$ are $\varrho_{\m}$-invariant and satisfy the defining relations for 
the graph algebra $C^*(W_1(\m))$. Thus, this assignment extends uniquely to a $*$-homomorphism 
$\phi:C^*(W_1(\m)) \to C^*(L_3)^{\varrho_{\m}}$. Injectivity of $\phi$ follows from 
\cite[Theorem 1.2]{Szy:ck}, since there are no closed paths in graph $W_3(\m)$ and $\phi(P_{w_j})\neq 0$ 
for all $j=0,\ldots,\tilde{m}_1$. 

It remains to verify that the map $\phi$ is surjective. First of all, $C^*(L_3)$ is a closed span of elements 
of the form $S_\alpha S_\beta^*$, where $\alpha$ and $\beta$ are two paths with the common range. 
The action $\varrho_{\m}$ rescales each such an element by a suitable power of $\xi$. Applying the 
conditional expectation from $C^*(L_3)$ onto $C^*(L_3)^{\varrho_{\m}}$ (integration over the orbits), 
we see that the fixed point algebra 
$C^*(L_3)^{\varrho_{\m}}$ is spanned by those elements $S_\alpha S_\beta^*$ which are fixed by 
$\varrho_{\m}$.  Hence it suffices to show that all such elements are in the range of $\phi$. 

If both $\alpha$ and $\beta$ end at $v_0$, then we must have $\alpha=\beta$ and thus $S_\alpha S_\beta^* 
=P_{v_0}=\phi(P_{w_0})+\sum_{j=2}^{\tilde{m}_1}\phi(P_{w_j})$. So suppose that $\alpha$ and $\beta$ 
end at $v_1$. If both $\alpha$ and $\beta$ contain only edges $e_{11}$, then again we must have $\alpha=
\beta$ and thus $S_\alpha S_\beta^* = P_{v_1} = \phi(P_{w_1})$. So we may assume that 
$\alpha = e_{00}^k e_{01}e_{11}^r$ for some $k,r\in\NN$. Now, if $\beta$ does not contain edge $e_{01}$, 
then $\beta=e_{11}^s$ for some $s\in\NN$ and we must have $m_0(k+1) + m_1r = m_1s$. This can 
only happen when $k=t\tilde{m}_1-1$ and $s=r+t\tilde{m}_0$ for some $t\in\NN\setminus\{0\}$. Since  $S_\alpha S_\beta^* = S_{\alpha'} S_{\beta'}^*$ with $\alpha'=e_{00}^ke_{01}$ and 
$\beta'=e_{11}^{s-r}$, we get $S_\alpha S_\beta^* = \phi(S_{f_{1(t-1)}})$ in this case. 

It remains to consider the case $\alpha = e_{00}^k e_{01}e_{11}^p$ and $\beta = e_{00}^le_{01}e_{11}^s$ 
for some $k,l,p,s\in\NN$. As above, from the start we may assume that $p=0$. We must have 
$m_0(k+1) = m_0(l+1) + m_1s$, and hence $\tilde{m}_0(k-l) = \tilde{m}_1s$. If $k=l$, then $s=0$ and 
$S_\alpha S_\beta^* = S_{e_{00}}^kS_{e_{01}}S_{e_{01}}^*(S_{e_{00}}^*)^k$. Write 
$k=r\tilde{m}_1+t$ with $r\in\NN$ and $t\in\{0,\ldots,\tilde{m}_1-1\}$. Then $S_\alpha S_\beta^*$ 
equals: (i) $\phi(P_{w_{t+2}})$ if $r=0$ and $t<\tilde{m}_1-1$, (ii) $\phi(S_{f_{(t+2)(r-1)}}
S_{f_{(t+2)(r-1)}}^*)$ if $r>0$ and $t<\tilde{m}_1-1$, (iii) $\phi(S_{f_{1r}}S_{f_{1r}}^*)$ if 
$t=\tilde{m}_1-1$. Note that this argument shows that for each path $\alpha$ in graph $L_3$ there exists 
an edge (or vertex) $f$ in graph $W_1(\m)$ such that $\phi(S_f)=P_\alpha$. 

Finally, suppose that $k\neq l$. It suffices to consider the case $k-l>0$, when also 
$s>0$. Then we must have $s=t\tilde{m}_0$ and $k-l=t\tilde{m}_1$ for some $t\in\NN\setminus\{0\}$, 
and thus $S_\alpha S_\beta = S_{e_{00}}^{l+t\tilde{m}_1}
S_{e_{01}}(S_{e_{11}}^*)^{t\tilde{m}_0}S_{e_{01}}^*(S_{e_{00}}^*)^l$. Let $f,h$ be edges (or 
possibly vertices) in graph $W_1(\m)$ such that  $\phi(f)=S_{e_{00}}^{l+t\tilde{m}_1}
S_{e_{01}}S_{e_{01}}^*(S_{e_{00}}^*)^{l+t\tilde{m}_1}$ and  $\phi(h) = S_{e_{00}}^l 
S_{e_{01}}S_{e_{11}}^{t\tilde{m}_0}(S_{e_{11}}^*)^{t\tilde{m}_0}S_{e_{01}}^*(S_{e_{00}}^*)^l$. 
Since $(l+t\tilde{m}_1)-l$ is a multiple of $\tilde{m}_1$, edges $f$ and $h$ have a common range. 
Then it is a bit tedious but not difficult to verify that $\phi(f)\phi(h)^* = S_{e_{00}}^{l+t\tilde{m}_1}
S_{e_{01}}(S_{e_{11}}^*)^{t\tilde{m}_0}S_{e_{01}}^*(S_{e_{00}}^*)^l$, and this completes the proof 
of surjectivity of $\phi$. 

As an immediate corollary of the isomorphism $C(\WW\PP_q^1(\m)) \cong C^*(W_1(\m))$, we obtain 
the following exact sequence:
$$ 
\xymatrix{  0 \ar[r] & \cK^{m_1/\gcd(m_0,m_1)} \ar[r] & C(\WW\PP_q^{1}(\m)) \ar[r] & \CC \ar[r] & 0.}
$$
It follows that $C(\WW\PP_q^1(\m))$ is an AF algebra and 
$$
K_0(C(\WW\PP_q^1(\m))) = \ZZ^{1+m_1/\gcd(m_0,m_1)}, \qquad K_1(C(\WW\PP_q^1(\m))) = 0,
$$
as required.
\end{proof}

Poposition~\ref{3dimweighted} contains full classification of algebras of continuous functions on the quantum weighted projective line: as a topological noncommutative space the quantum projective line $\WW\PP_q^1(m_0, m_1)$  is isomorphic to the quantum teardrop $\WW\PP_q^1(1, m)$, where  $m=m_1/\gcd(m_0,m_1)$.  

\begin{proposition}\label{prop.k.qwps}
Let $\m:=m_0, \ldots , m_n$ be positive integers such that there exists  $j\in\{0,1,\ldots, n -1\}$ so that $m_j$ is 
relatively prime with $m_n$. Then there exists an exact sequence
\begin{equation}\label{ses.qwps}
\xymatrix{  0 \ar[r] & \cK^{m_n} \ar[r] & C(\WW\PP_q^{n}(\m)) \ar[r] & C(\WW\PP_q^{n-1}(\m)) \ar[r] & 0.}
\end{equation}
\end{proposition}
\begin{proof}
We use the identification $ C(\WW\PP_q^{n}(\m))\cong C^*(L_{2n+1})^{\varrho_{\m}}$. Let $J$ be the 
closed span of all $S_\alpha S_\beta^*\in C^*(L_{2n+1})^{\varrho_{\m}}$ such that $\alpha,\beta$ are 
paths in $L_{2n+1}$ with both $\alpha$ and $\beta$ ending at vertex $m_n$. Then $J$ is a closed, 
two-sided ideal of $C^*(L_{2n+1})^{\varrho_{\m}}$ such that the quotient $C^*(L_{2n+1})^{\varrho_{\m}}
/J$ is isomorphic to $C^*(L_{2n-1})^{\varrho_{\m}}\cong C(\WW\PP_q^{n-1}(\m))$. Thus it suffices to show 
that $J\cong\cK^{m_n}$. 

For each $k\in\{0,1,\ldots,m_n -1\}$ let $J_k$ be the closed, two-sided ideal of $J$ generated by all 
projections $S_\alpha S_\alpha^*$ such that $\alpha$ is a path in $L_{2n+1}$ ending at $v_k$ and 
$\varrho_{\m}(S_\alpha)=\xi^l S_\alpha$ with $l\equiv k\, (\hspace{-3mm}\mod m_n)$. We claim that 
$J_k\cong\cK$ and $J_k J_r=\{0\}$ for $k\neq r$. Indeed, let $S_\alpha S_\alpha^*$ and $S_\beta S_\beta^*$ 
be two projections in $J_k$, as above. Then for a suitable integer $t$ the element $S_\alpha S_{\m_n}^t 
S_\beta^*$ is a partial isometry in $J_k$ with domain $S_\beta S_\beta^*$ and range $S_\alpha S_\alpha^*$. 
On the other hand, if $S_\alpha S_\alpha^*\in J_k$ and $S_\beta S_\beta^*\in J_r$ with $k$ not congruent 
to $r$ modulo $m_n$, then these two projections are not equivalent in $J$, since there is no $t\in\ZZ$ for 
which   $S_\alpha S_{k}^t S_\beta^*$ is in the fixed point algebra $C^*(L_{2n+1})^{\varrho_{\m}}$. 
It remains to show that $J_k\neq\{0\}$ for each $k$. Let $j<m_n$ be such that $m_j$ and $m_n$ are 
relatively prime. Consider path $\alpha=e_{jj}^t e_{j(j+1)}e_{(j+1)(j+2)}\ldots e_{(m_n -1)m_n}$.  
Since $m_j$ and $m_n$ are relatively prime, for each $k$ we can find a positive integer $t$ such that 
$tm_j + \sum_{i=j}^{m_n -1} m_i $ is congruent to $k$ modulo $m_n$. 
\end{proof}

\begin{corollary}
Let $\m:=m_0, \ldots , m_n$ be a sequence of positive integers such that for each $j\geq 1$ there is an 
$i<j$ so that $m_i$ and $m_j$ are relatively prime. Then 
$$
K_0(C(\WW\PP_q^n(\m))) = \ZZ^{1+\sum_{i=1}^n m_i }, \qquad K_1(C(\WW\PP_q^n(\m))) = 0.
$$
\end{corollary}
\begin{proof}
We proceed by induction on $n$. Case $n=1$ being contained in Proposition~\ref{3dimweighted}. 
Applying the $K$-functor to \eqref{ses.qwps} we obtain the six-term exact sequence
$$
 \xymatrix{ K_0(\cK^{m_n}) \ar[r] &  K_0(C(\WW\PP_q^{n}(\m))) \ar[r] & K_0(C(\WW\PP_q^{n-1}(\m)))
\ar[d] \\
K_1(C(\WW\PP_q^{n-1}(\m)))  \ar[u] & \ar[l] K_1(C(\WW\PP_q^{n}(\m))) & \ar[l] K_1(\cK^{m_n}) \, .}
 $$
 Since outer terms in the bottom row vanish (the left one by inductive assumption) also the middle term is 0, as required. Thus again using the inductive assumption and the K-theory of compact operators we obtain a short exact sequence
 $$
\xymatrix{  0 \ar[r] & \ZZ^{m_n} \ar[r] & K_0(C(\WW\PP_q^{n}(\m))) \ar[r] & 
\ZZ^{1+\sum_{i=1}^{n-1} m_i }  \ar[r] & 0,}
$$
which splits as a sequence of abelian groups thus confirming the stated form of $K$-groups of quantum weighted projective spaces.
\end{proof}

 \section*{Acknowledgments}
The first named author would like to extend his warmest thanks to the members of IMADA, University of Southern Denmark, Odense, where the work on this paper was partly carried out in May and December 2015. The research of the second named author was partially supported by  the FNU Project Grant 
`Operator algebras, dynamical systems and quantum information theory' (2013--2015),
 the Villum Fonden Research Grant `Local and global structures of groups and their algebras' (2014--2018), 
and by the Mittag-Leffler Institute during his stay there in January-February, 2016.

\end{document}